\documentclass[a4paper,11pt,leqno]{amsart}
\usepackage[applemac]{inputenc}
\usepackage{epsfig,graphics,xypic,url}
\usepackage{tikz}

\usepackage{amsthm,amssymb,amsbsy,a4wide,verbatim}
\usepackage{bbm}
\usepackage{longtable}

\def\Z{\mathbbm{Z}}

\def\N{\mathbbm{N}}

\def\C{\mathbbm{C}}
\def\Q{\mathbbm{Q}}

\def\ii{\mathrm{i}}

\def\redu{\mathrm{red}}

\def\queue{\mathrm{tail}}

\newtheorem{theor}{Theorem}

\newtheorem{prop}[theor]{Proposition}

\newtheorem{definition}[theor]{Definition}
\numberwithin{theor}{section}

\makeatletter
\newcommand{\xleftrightarrow}[2][]{\ext@arrow 3359\leftrightarrowfill@{#1}{#2}}
\newcommand{\xdashrightarrow}[2][]{\ext@arrow 0359\rightarrowfill@@{#1}{#2}}
\newcommand{\xdashleftarrow}[2][]{\ext@arrow 3095\leftarrowfill@@{#1}{#2}}
\newcommand{\xdashleftrightarrow}[2][]{\ext@arrow 3359\leftrightarrowfill@@{#1}{#2}}
\def\rightarrowfill@@{\arrowfill@@\relax\relbar\rightarrow}
\def\leftarrowfill@@{\arrowfill@@\leftarrow\relbar\relax}
\def\leftrightarrowfill@@{\arrowfill@@\leftarrow\relbar\rightarrow}
\def\arrowfill@@#1#2#3#4{%
  $\m@th\thickmuskip0mu\medmuskip\thickmuskip\thinmuskip\thickmuskip
   \relax#4#1
   \xleaders\hbox{$#4#2$}\hfill
   #3$%
}
\makeatother

\title{Proof of the BMR conjecture for $G_{20}$ and $G_{21}$}
\author{Ivan Marin}
\address{LAMFA \\ UMR CNRS 7352 \\ Universit\'e de Picardie-Jules Verne \\ 33 rue Saint-Leu \\ 80039 Amiens Cedex 1 \\ France }
\email{ivan.marin@u-picardie.fr} 

\date{January 31, 2017.}
\begin{document}

\begin{abstract} We prove two new cases of the Brou\'e-Malle-Rouquier freeness conjecture for the Hecke algebras associated
to complex reflection groups. These two cases are the complex
reflection groups of rank 2 called $G_{20}$ and $G_{21}$ in the Shephard and Todd classification. This reduces the number of remaining
unproven cases to 3.
\end{abstract}

\maketitle

\section{Introduction}

Two decades ago, M. Brou\'e, G. Malle and R. Rouquier conjectured in \cite{BMR}
that the generalized Hecke algebras that they attached to an arbitrary complex
reflection group satisfy the crucial structural property of the ordinary (Iwahori-)Hecke algebras
attached to a finite Coxeter group, namely that they are free modules of rank
equal to the order of the group. This is known as the BMR \emph{freeness conjecture},
and it can be easily reduced to the case where the complex reflection
group $W$ is irreducible.  We refer to \cite{CYCLO} for a general exposition
of this conjecture and standard results about it.

The Shephard-Todd classification of irreducible complex reflection groups defines an infinite
family $G(de,e,n)$ of such groups, for which the conjecture was already known to hold
by work of Ariki and Ariki-Koike (see \cite{ARIKI,ARIKIKOIKE}), and a long
list of exceptional groups. Subsequent works have proved it for most
of the exceptional groups, notably all the ones of rank at least 3 (see
\cite{CYCLO,HECKECUBIQUE,MARINPFEIFFER}),
and most of the ones of rank 2 (see \cite{CHAVLITHESE,CHAVLICRAS}). In rank 2,
the 5 remaining ones are named, in Shephard-Todd notation, $G_{17}$, $G_{18}$,
$G_{19}$, $G_{20}$ and $G_{21}$. 
In this work, we prove the cases of $G_{20}$ and $G_{21}$, by a
method of a different nature than in the previous works. This reduces the list of remaining
cases to the 3 groups $G_{17}$, $G_{18}$ and $G_{19}$, for which it appears difficult
to apply readily the methods of this paper.

In section 2 we recall the main definitions, and prove a technical property that
will allow us to work over rings of definitions which are polynomial rings,
instead of the usual Laurent polynomial rings. In section 3 we explain the
general method : how we find a potential basis for the Hecke algebras and how
we find a list of rewriting rules. Then, sections 4 and 5 contain the rewriting
rules we used in the cases of $G_{20}$ and $G_{21}$, respectively.

\bigskip

The GAP4 programs used for $G_{21}$ can be found on my webpage \url{http://www.lamfa.u-picardie.fr/marin/G20G21code-en.html}.

\bigskip

{\bf Acknowledgements.} I thank G. Pfeiffer for improving (optimizing) my original programs.

\section{Definitions and preliminaries}

Let $W$ be a finite complex (pseudo-)reflection group.
We let $B$ denote the braid group of $W$, as defined in \cite{BMR} \S 2 B, and recall that
a (pseudo-)reflection $s$ is called distinguished if its only nontrivial eigenvalue is $\exp(2 \ii \pi/o(s))$, where
$\ii \in \C$ is the chosen square root of $-1$ and $o(s)$ denotes the order of $s \in W$.

We  let $R =  
\Z[ a_{s,i}, a_{s,0}^{-1}] $ where $s$ runs
over the distinguished reflections
in $W$ and $0 \leq i \leq o(s)-1$, where $o(s)$ is the order of $s$ in $W$,
with the convention $a_{s,i} = a_{s',i}$ if $s,s'$ are conjugates in $W$.
For the standard notion of a braided reflection associated to $s$ we refer to \cite{BMR},
where they are described as `generators-of-the-monodromy' around the
divisors of the orbit space. The definition of the Hecke algebra associated to $W$ reads
as follows.

\begin{definition} The generic Hecke algebra is the quotient of the group algebra $R B$ by the relations
$\sigma^{o(s)} - a_{s,o(s)-1}\sigma^{o(s)-1} - \dots - a_{s,0} = 0$ for each braided reflection $\sigma$ associated
to $s$.
\end{definition}

Actually,
it is enough to choose one such relation per conjugacy class of distinguished reflection, as all the
corresponding braided reflections are conjugates in $B$. Although we are not going to use this result in our proof, we mention that it was already known by work of Etingof and Rains
(see \cite{ETINGOFRAINS})
that the Hecke algebras of the groups considered here are modules of finite type.
Our main result can now
be stated as follows.

\begin{theor} \label{theo:main}
When $W$ is a complex reflection group of Shephard-Todd type $G_{20}$
or $G_{21}$, then the generic Hecke algebra of $W$ is a free $R$-module of rank $|W|$.
\end{theor}

Let $R_0 = \Z[b_{s,i}, 1 \leq i \leq o(s)]$ where $s$ runs over the distinguished reflections,
with the convention $b_{s,i} = b_{s',i}$ if $s,s'$ are conjugates in $W$,
and define $H_0$ as the quotient of $R_0 B$ by the relations 
$$\sigma^{o(s)} - b_{s,o(s)-1}\sigma^{o(s)-1} - \dots 
- b_{s,1}\sigma- 1 = 0$$ 
for each braided reflection $\sigma$ associated
to $s$. Again, it is enough to choose one such relation per conjugacy class of distinguished reflection. We let $H$ denote the usual Hecke algebra, defined over $R$.

The next proposition is useful in order to reduce the number of parameters involved in
the computations.

\begin{prop} \label{prop:R0} 
{\ }
\begin{enumerate}
\item $H_0$ is spanned by $|W|$ elements as a $R_0$-module iff it is a free $R_0$-module of rank $|W|$.
\item $H$ is a free $R$-module of rank $|W|$ iff $H_0$ is a free $R_0$-module of rank $|W|$.
\end{enumerate}
\end{prop}
\begin{proof} The proof of (i) is the same as the one of \cite{CYCLO}, proposition 2.4. We prove (ii). We have a ring
morphism $\phi_1 : R \to R_0$ defined by $a_{s,i} \mapsto b_{s,i}$ if $i \geq 1$, $a_{s,0} \mapsto 1$,
for which $H_0 = H \otimes_{\phi_1} R_0$. Therefore, if $H$ is a free $R$-module of rank $|W|$,
we get the $H_0 \simeq R^{|W|} \otimes_{\phi_1} R_0 \simeq R_0^{|W|}$ is also free of rank $|W|$.
We prove the converse. Assume that $H_0$ is $R_0$-free of rank $|W|$. Let $A = \Z[x_s,x_s^{-1}]$ where $s$ runs
among the distinguished reflections of $W$ with $x_s = x_{s'}$ if $s,s'$ are conjugates in $W$. We have an injective ring
morphism $R \to A \otimes_{\Z} R_0$ defined by $a_{s,0} \mapsto x_s^{o(s)} = x_s^{o(s)} \otimes 1$,
and $a_{s,i} \mapsto b_{s,i}x_s^{o(s)-i} = x_s^{o(s)-i}  \otimes b_{s,i}$ for $i \geq 1$. We first note that $A \otimes R_0$
is a free $R$-module of finite rank, since it is easily checked that
$$
A \otimes R_0 = \bigoplus_{s \in \mathcal{S}} \bigoplus_{0 \leq i < o(s)} x_s^i R
$$
where $\mathcal{S}$ is a system of representatives of the conjugacy classes of distinguished reflections.

We denote $\check{H}_0$ the quotient of
the group algebra $(A \otimes_{\Z} R_0)B$ of $B$ over $A \otimes_{\Z} R_0$
by the relations $(x_s \sigma)^{o(s)} 
- b_{s,o(s)-1}x_s (x_s \sigma)^{o(s)-1} - \dots - b_{s,1}x_s^{o(s)-1} (x_s \sigma) - x_s^{o(s)}= 0$ for each braided reflection $\sigma$ associated
to $s$. 
We consider the composite map 
$$
\xymatrix{
A B \ar[r]_{\!\!\!\!\!\!\!\!\! \!\!\!\!\!\!\!\!\! \Delta}& (A B) \otimes_A (A B) \ar[r]_{\mathrm{Id} \otimes Ab} & (A B)  \otimes_A (A B^{ab}) 
\ar[rr]_{\ \ \ \ \ \ \ \mathrm{Id} \otimes (s \mapsto x_s)} & \ \ &(A B) \otimes_A A \ar[r]_{\ \ \ \ \simeq} & A B
}
$$
where $\Delta$ is the usual coproduct of the Hopf algebra $A B$, $Ab : B \to B^{ab}$ the abelianization
morphism and, by abuse of notations, the associated linear map $A B \to A B^{ab}$, and `$s \mapsto x_s$'
denotes the map $B^{ab} \to A$ defined as follows. It is known (see e.g. \cite{BMR}) that $B^{ab}$ is a free $\Z$-module admitting a natural
basis indexed by the conjugacy classes of distinguished reflections. The map is defined by mapping the basis
element associated to (a conjugacy class of) distinguished reflection $s$ to the scalar $x_s \in A$.

The composite map is easily checked to be an $A$-algebra isomorphism. Its natural extension
$(A \otimes R_0) B \to (A \otimes R_0) B$ induces an isomorphism $\check{H}_0  = H \otimes_R (A \otimes_{\Z} R_0)$.

Now, if $H_0$ is $R_0$-free of rank $|W|$, then $\check{H}_0 = H_0 \otimes_{R_0} A$ is $A \otimes R_0$-free of rank $|W|$.
Since $A \otimes R_0$ is a free $R$-module of finite rank, this implies that $\check{H}_0$ is a free $R$-module of finite
rank, and also that, since $\check{H}_0  = H \otimes_R (A \otimes_{\Z} R_0)$, that the $R$-module $H$ is a direct factor
of $\check{H}_0$. Therefore $H$ is projective as a $R$-module and this implies that $H$ is free of rank $|W|$ by \cite{CYCLO}, proposition 2.5.

\end{proof}

The groups we hare interested in are the ones denoted $G_{20}$ and $G_{21}$ in the Shephard-Todd notation.
They admit presentations symbolized by the following diagrams

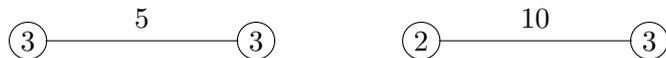
\begin{center}
\begin{tikzpicture}
\draw (0,0) node (1) {$3$};
\draw (1)++(3,0) node (2) {$3$};
\draw (1) edge (2);
\draw (1) circle (7pt);
\draw (2) circle (7pt);
\draw (1)++(1.5,.3) node {$5$};
\end{tikzpicture}
\ \ \ 
\ \ \ 
\ \ \ 
\ \ \ 
\begin{tikzpicture}
\draw (0,0) node (1) {$2$};
\draw (1)++(3,0) node (2) {$3$};
\draw (1) edge (2);
\draw (1) circle (7pt);
\draw (2) circle (7pt);
\draw (1)++(1.5,.3) node {$10$};
\end{tikzpicture}
\end{center}
that is $G_{20} = \langle s_1, s_2 \ | \ s_1 s_2 s_1 s_2 s_1 = s_2 s_1 s_2 s_1 s_2, \ s_1^3=s_2^3 =1 \rangle$
and $G_{21} = \langle s_1, s_2 \ | \ (s_1 s_2)^5= (s_2 s_1)^5, \ s_1^2=s_2^3 =1 \rangle$. In these
presentations, $s_1,s_2$ are distinguished reflections, and every distinguished reflection is a conjugate of one of them.
Moreover, $s_1$ and $s_2$ are conjugates in $G_{20}$, as is readily deduced from the presentation itself.
The corresponding braid groups admit the same
presentations, with the order relations removed.

We use the above proposition to define the Hecke algebras of $G_{20}$ and $G_{21}$ over $R_0$,
where $R_0 = \Z[a,b]$ for $G_{20}$ and $R_0 = \Z[a,b,q]$ for $G_{21}$, with relations
$$
\begin{array}{|cclcl|cclcl|}
\hline
G_{20} & : & s_1^3 &=& a s_1^2 + b s_1 + 1 & G_{21} &: & s_1^2 &=& q s_1 + 1 \\
 &&s_2^3 &=& a s_2^2 + b s_2 + 1  & & & s_2^3 &=& a s_2^2 + b s_2 + 1 \\
 \hline
 \end{array}
 $$
 In the subsequent section we prove these Hecke algebras are spanned by the `right'
 number of elements, and this proves theorem \ref{theo:main} by proposition \ref{prop:R0}.

\section{General method}

In this section we describe the general method we used to prove the conjecture in these cases. It proceeds in several steps.

\begin{enumerate}
\item Heuristics/Experimentation
\item Incremental determination of computational rules
\item Right multiplication table
\end{enumerate}
\subsection{Heuristics/Experimentation}

The first crucial element is of heuristic nature, provided by a software able to
compute non-commutative Gr\"obner basis for finitely presented associative $\Q$-algebras.
We used the GAP4 package GBNP (see \cite{GBNP}) with the standard (`\verb+deglex+') ordering
for monomials, taking as input the presentations of \cite{BMR},
where we specialized the Hecke algebras at more or less random parameters. For $G_{20}$
and $G_{21}$ it finished in reasonable time for all the specializations we tried,
while for $G_{18}$ and $G_{19}$ it was not able to complete the computation
after several months of running time, except for the simple case of the group algebra
specialization, that is the presentation of $W$ viewed as a presentation of the
Hecke algebra at very special parameters. For all the groups of the so-called icosahedral
series of complex reflection groups of rank $2$, GBNP nevertheless finds a Gr\"obner basis of the rational group algebra of $W$.

It turns out that most if not all the specializations we tried for $G_{20}$ and $G_{21}$
(including the group algebra specialization) provided the same number of elements
for the Gr\"obner basis. As an indication of the complexity of this heuristic data, we provide
the following table, were $\# W$ is the order of $W$ and $\# gb$ is the number
of elements in the Gr\"obner basis. The groups whose name appears in bold fonts are the ones for
which the BMR freeness conjecture is now proved, after work of Chavli for $G_{16}$
(see \cite{CHAVLITHESE,CB3}), 
of Marin-Pfeiffer for $G_{22}$ (see \cite{MARINPFEIFFER}), and by the present work for $G_{20}$ and $G_{21}$.

The output of GBNP we are interested in is the collection $\mathcal{G}$ of leading monomials
of the Gr\"obner basis. In case we had computed the Gr\"obner basis
for several specializations this collection turned out to be independent of the
specialization. From this one computes easily the set $\mathcal{B}$
of all words avoiding the patterns which belong to $\mathcal{G}$. As expected, it has cardinality
$|W|$ and provides for these specializations a basis of the Hecke algebra.

$$
\begin{array}{|c|c|c||c|c|c|}
\hline
W & \# W & \# gb & W & \# W & \# gb  \\
\hline
\hline
\mathbf{G_{16}} & 600 & 44 &  \mathbf{G_{20}} & 360 & 36\\
\hline
G_{17} & 1200 & 49 &  \mathbf{G_{21}} & 720 & 30 \\
\hline
G_{18} & 1800 & 138 &  \mathbf{G_{22}} & 240 & 66 \\
\hline
G_{19} & 3600 & 558 & & & \\
\hline
\end{array}
$$

\subsection{Incremental determination of computational rules}

It so happens that all defining relations are included in the Gr\"obner bases
provided by GBNP.
We view these as the first step in the construction of an \emph{ordered list} $\mathcal{L}$
of rewriting rules of the form $w \leadsto c_w$ where $w \in \mathcal{G}$
and $c_w$ is a $R_0$-linear combination of elements of $\mathcal{B}$, with the
property that the equality $w = c_w$ holds inside the Hecke algebra $H_0$.
More precisely, the defining relations of the braid groups of the form $b_1 = b_2$
are included under the form $b_1 \leadsto b_2$ for $b_1 > b_2$. One checks that $b_1 \in
\mathcal{G}$ and $b_2 \in \mathcal{B}$ in all cases. The order relations,
of the form $\sigma^m = b_{s,m-1}\sigma^{m-1}+\dots+ b_{s,1}\sigma+1$,
are also included under the form $\sigma^m \leadsto b_{s,m-1}\sigma^{m-1}+\dots+ b_{s,1}\sigma+1$. We denote $\mathcal{L}_0$ the ordered list of leading terms $w \in \mathcal{G}$
of the rules in $\mathcal{L}$.

The incremental process aims at enlarging $\mathcal{L}$ so that
$\mathcal{L}$ contains at the end as many elements as $\mathcal{G}$, with the set
of elements inside $\mathcal{L}_0$ being equal to $\mathcal{G}$.

The way we enlarge $\mathcal{L}$ is as follows. We use an algorithm for
computing a given word as a $R_0$-linear combination
of words as follows.
\begin{itemize}
\item Input : a word $w$ in the generators and their inverses
\item If $w$ contains the inverse of a generator, replace $w$ by a
linear combination of \emph{positive} words, by applying the rewriting rules
$\sigma^{-1} \leadsto \sigma^{m-1} -b_{s,m-2}\sigma^{m-1}-\dots- b_{s,1}$
as many times as needed, and apply the present algorithm to these
words.
\item If $w \in \mathcal{B}$, then return $w$.
\item If not, then look for the first element in $\mathcal{L}_0$
which appear as a subword in $w$. If there is none, return \verb+fail+.
If there is one $v$, with $w = avb$, then replace it with the linear combination $ac_vb$, where $v \leadsto c_v$ belongs to $\mathcal{L}$, and apply the algorithm
to each monomial of this linear combination.
\end{itemize}

It is clear that, if the present algorithm terminates for a given word $w$,
producing a $R_0$-linear combination $b_w$, then the equality $w =b_w$
holds inside $H_0$. Adding more elements in $\mathcal{L}$
will not change the result if the input is one for which the algorithm already
terminated, but instead potentially increases the number of words
for which it does provide a result.

Our strategy is then to establish a number of equalities inside $H_0$
of the form $w_i = b_{w_i}$, where $w_i \in \mathcal{G}$ and $b_{w_i}$
is a linear combination of words with possibly negative powers,
such that $\mathcal{L}$ originally contains the first $w_1,\dots,w_{n_0}$
originating from the defining relations, and so that we can build incrementally 
$\mathcal{L}$ as follows.
\begin{itemize}
\item If $\mathcal{L} = (w_1 \leadsto c_{w_1},\dots, w_n \leadsto c_{w_n})$,
then apply the algorithm with $\mathcal{L}$ to $b_{w_{n+1}}$. It produces
a linear combination $c_{w_{n+1}}$. Add to $\mathcal{L}$ the rule
$w_{n+1} \leadsto c_{w_{n+1}}$.
\item Start again with the new $\mathcal{L}$.
\end{itemize}

For the first group ($G_{20}$) we are interested in, we managed to produce
a convenient list 
of rewriting rules $w_i \leadsto b_{w_i}$ completely
by hand (see section \ref{sect:rulesG20}). For the group $G_{21}$ the making
of this list had to be partly automatized, too (see section \ref{sect:rulesG21}).

\subsection{Right multiplication table}

Completing the (right)multiplication table is then
merely a way to check that $H_0$ is indeed spanned by the elements of
$\mathcal{B}$. It is sufficient to calculate, using the algorithm described
in the previous subsection, each word $ws$ where $w \in \mathcal{B}$
and $s$ a generator, as a $R_0$-linear combination of the
words in $\mathcal{B}$.

\section{Rules for $G_{20}$}
\label{sect:rulesG20}
We first provide the list of rewriting rules, and subsequently justify it.
$$
\begin{array}{clcl}
(1) & 111 & \leadsto & a. 11 + b. 1 + \emptyset \\
(2) & 222 & \leadsto & a. 22 + b. 2 + \emptyset \\
(3) & 21212 & \leadsto & 12121 \\
(4) & 2112121 & \leadsto & 1212112 \\
(5) & 2121122& \leadsto & 1122121+(a).212112+(-a).122121+(-b).22121+(b).21211 \\
(6) & 22122121& \leadsto & 12122122+(a).2122121+(-a).1212212+(b).122121+(-b).121221 \\
(7) & 2211212& \leadsto & 1212211+(a).211212+(-a).121221+(-b).12122+(b).11212 \\
(8) & 21211211& \leadsto & 11211212+(a).2121121+(-a).1211212+(b).212112+(-b).211212 \\
(9) & 212112122& \leadsto & 112122121+(a).21211212+(-a).12122121+(-b).2122121+(b).2121121 \\
(10) &221211212& \leadsto & 121221211+(a).21211212+(-a).12122121+(-b).1212212+(b).1211212 \\ 
\end{array}$${}
$$
\begin{array}{clcl}
(11) & 21221122 & \leadsto & a.2121122 + b.211122 + a.21\bar{2}122+ b . 212 + a. 21\bar{2}\bar{1}2 + b. 21\bar{2}\bar{1} + 
\bar{1}\bar{2}\bar{1} 21 \\
(12) & 22112212 & \leadsto & a.2112212+b.112212+a.\bar{2}12212+b.212+a.\bar{2}\bar{1}212+b.\emptyset+12\bar{1}\bar{2}\bar{1}\\
(13) & 2112122121 & \leadsto & a.2212122121\bar{2}+b.221212121\bar{2}+22121121\bar{2} \\
(14) & 212212211 & \leadsto & a.21221221+b.2122122+ a.212212\bar{1}+b.2122+ 212\bar{1}\bar{2}\bar{1}212\\
(15) & 211221122 & \leadsto & a.21122112+b.2112211+a.211221\bar{2}+b.2112+a.2112\bar{1}\bar{2}+b.21\bar{2} \\
& & & +a.21\bar{2}\bar{1}\bar{2}+ b.\bar{1}\bar{2} + \bar{1} \bar{2}\bar{1} \bar{2}1\\
(16) & 221122112 & \leadsto & (a).21122112+(b).1122112+(a).\bar{2}122112+(b).2112+(a).\bar{2}\bar{1}2112 \\
& & & +(b).\bar{2}12+(a).\bar{2}\bar{1}\bar{2}12+(b).\bar{2}\bar{1}+1\bar{2}\bar{1}\bar{2}\bar{1}\\
(17) & 2112112211 & \leadsto & (a).211212211+(b).21122211+(a).2112\bar{1}211+(b).21121 \\
& & & +(a).2112\bar{1}\bar{2}1+(b).2112\bar{1}\bar{2}+21\bar{2}\bar{1}\bar{2}12 \\ 
(18) & 221121121 & \leadsto & (a).21121121+(b).1121121+(a).\bar{2}121121+(b).1121+121\bar{2}\bar{1}\bar{2}121 \\
(19) & 21221121122 & \leadsto & (a).2122112122+(b).212211222+(a).2122112\bar{1}2\\
& & & +(b).2122112\bar{1}+212\bar{1}\bar{2}\bar{1}21121 \\
(20) & 21122121121 & \leadsto & (a).2122121121+(b).222121121+(a).2\bar{1}2121121\\
& & & +(b).221121+2212212\bar{1}\bar{2} \\
\end{array}
$$
{}
$$
\begin{array}{clcl}
(21) & 22112112212 & \leadsto & (a).2211212212+(b).221122212+(a).22112\bar{1}212\\
& & & +(b).221122+22112212\bar{1}\bar{2}\bar{1} \\
(22) & 211211211212 & \leadsto & 2112212122212\bar{1} \\
(23) & 211211212212 & \leadsto & 2112121212\bar{1}212 \\
(24) & 211212211211 & \leadsto & 21\bar{2}12121211211 \\
(25) & 2112112212211 & \leadsto & (a).211211221211+(b).21121122111+(a).211211221\bar{2}1\\
& & & +(b).211211221\bar{2}+21\bar{2}\bar{1}2112212 \\
(26) & 2112122122122 & \leadsto & 2121212\bar{1}2122122 \\
(27) & 2112211211221 & \leadsto & (a).212211211221+(b).22211211221+(a).2\bar{1}211211221\\
& & & +(b).21211221+(a).2\bar{1}\bar{2}1211221+(b).21221+2211221\bar{2}\bar{1}\bar{2} \\
(28) & 2112212212212 & \leadsto & (a).212212212212+(b).22212212212+(a).2\bar{1}212212212\\&&&+(b).22212212+2212\bar{1}\bar{2}\bar{1}212212 \\
(29) & 2122122122122 & \leadsto & \bar{1}21212\bar{1}2122122122 \\
(30) & 2212212212212 & \leadsto & 2212212212\bar{1}21212\bar{1} \\
\end{array}
$$
{}
$$
\begin{array}{clcl}
(31) & 21121121121121 & \leadsto & 21121121121\bar{2}12121\bar{2} \\
(32) & 21121121121122 & \leadsto & 
a.2112112112122+b.211211211222\\
& & & +a.2112112112\bar{1}2+b.2112112112\bar{1}+21121121\bar{2}\bar{1}\bar{2}
121 \\
(33) & 21121121122122 & \leadsto & 21\bar{2}12121\bar{2}121122122 \\
(34) & 21122122112112 & \leadsto & 2112\bar{1}21212\bar{1}2112112 \\
(35) & 211211221221221 & \leadsto & 21\bar{2}12121\bar{2}1221221221 \\
(36) & 2112112112212112 & \leadsto & 211211211212121\bar{2}12 \\
\end{array}
$$

We now justify each one of the above rules.
Rule (3) is a direct consequence of the braid relation,
and (4) follows from 
$21\underline{12121}= \underline{21212}12
= 1212112$.

We have
$$
\begin{array}{lcl}
2121122 &=& \bar{1} \underline{12121}122 \\
&=& \bar{1} 2\underline{12121}22 \\
&=& \bar{1} 22121\underline{222} \\
&=& a.\bar{1} 2212122+b.\bar{1} 221212 + \bar{1} 22121 \\
&=& a.\bar{1} 2\underline{21212}2+b.\bar{1} 2\underline{21212} + \bar{1} (111 - a. 11 - b.1)22121 \\
&=& a.\bar{1} \underline{21212}12+b.\bar{1} \underline{21212}1 +  (11 - a. 1 - b.\emptyset)22121 \\
&=& a.\bar{1} 1212112+b.\bar{1} 121211 +  1122121 - a. 122121 - b.22121 \\
&=& a.212112+b.21211 +  1122121 - a. 122121 - b.22121 \\
\end{array}
$$
whence (5).

We have 
$
22122121 = 
2212\underline{21212}\bar{2} = 
2\underline{21212}121\bar{2} = 
\underline{21212}1121\bar{2} = 
1212\underline{111}21\bar{2}$ hence
$$
\begin{array}{lcl}
22122121 &=&1212(a.11 + b.1 + \emptyset)21\bar{2} \\
&=&a.\underline{12121}121\bar{2}+b.12\underline{12121}\bar{2}+121221\bar{2} \\
&=&a.212\underline{12121}\bar{2}+b.1221212\bar{2}+121221\bar{2} \\
&=&a.21221212\bar{2}+b.122121+121221(22 - a.2 - b\emptyset) \\
&=&a.2122121+b.122121+12122122 -a. 1212212-b. 121221\\
\end{array} 
$$
whence (6).

We have $2211212 = 221\underline{12121} \bar{1}
= 22\underline{12121}2 \bar{1}
= \underline{222}12122 \bar{1}$
hence
$$
\begin{array}{lcl}
2211212 &=&(a.22 + b.2 + \emptyset)12122 \bar{1} \\
&=&a.2\underline{21212}2 \bar{1} + b.\underline{21212}2 \bar{1} +12122 \bar{1} \\
&=&a.21\underline{21212} \bar{1} + b.1\underline{21212} \bar{1} +1212211 - a.121221 - b.12122 \\
&=&a.2112121 \bar{1} + b.112121 \bar{1} +1212211 - a.121221 - b.12122 \\
&=&a.211212 + b.11212 +1212211 - a.121221 - b.12122 \\
\end{array}
$$
whence (7).

We have $21211211 = \bar{1} \underline{12121}1211 = \bar{1} 21\underline{21212}11 = \bar{1} 211212\underline{111}$
hence
$$
\begin{array}{lcl}
21211211 &=&\bar{1} 211212(a.11 + b.1 + \emptyset) \\
&=&a \bar{1} 21\underline{12121}1+b\bar{1} 21\underline{12121}+(11-a.1 -b. \emptyset) 211212  \\
&=&a \bar{1} \underline{21212}121+b\bar{1} \underline{21212}12+11211212-a.1211212 -b.211212   \\
&=&a \bar{1} 12121121+b\bar{1} 1212112+11211212-a.1211212 -b.211212   \\
&=&a 2121121+b212112+11211212-a.1211212 -b.211212   \\
\end{array}
$$
and this proves (8). Similarly, we have $212112122 = \bar{1}\underline{12121}12122 = \bar{1} 212\underline{12121}22
= \bar{1} 2122121\underline{222}$ hence
$$
\begin{array}{lcl}
212112122 &=& \bar{1} 2122121(a. 22 + b.2 + \emptyset) \\
 &=& a.\bar{1} 212\underline{21212}2+b.\bar{1} 212\underline{21212}+ \bar{1} 2122121 \\
 &=& a.\bar{1} \underline{21212}1212+b.\bar{1} \underline{21212}121+ (11 - a.1 - b) 2122121 \\
 &=& a.\bar{1} 121211212+b.\bar{1} 12121121+ 112122121 - a.12122121 - b. 2122121 \\
 &=& a.21211212+b.2121121+ 112122121 - a.12122121 - b .2122121 \\
\end{array}
$$
and this proves (9).
Finally we have $221211212 = 22121\underline{12121} \bar{1}
= 22\underline{12121}212 \bar{1}
= \underline{222}1212212 \bar{1}$ hence
$$
\begin{array}{lcl}
221211212 &=& (a.22 + b.2 + \emptyset)1212212 \bar{1} \\
 &=& a.2\underline{21212}212 \bar{1} + b.\underline{21212}212 \bar{1} + 1212212 \bar{1} \\
 &=& a.2121\underline{21212} \bar{1} + b.121\underline{21212} \bar{1} + 1212212 (11 - a.1 - b) \\
 &=& a.212112121 \bar{1} + b.12112121 \bar{1} +  121221211 - a.12122121 - b.1212212 \\
 &=& a.21211212 + b.1211212 +  121221211 - a.12122121 - b.1212212 \\
\end{array}
$$
and this proves (10).

We have $21\underline{22}1122 = a.2121122+b.211122+21\bar{2}1122$
and $21\bar{2}\underline{11}22 = a.21\bar{2}122 + b . 21\bar{2}22  + 21\bar{2}\bar{1}22
=  a.21\bar{2}122 + b . 212  + 21\bar{2}\bar{1}22$.
Then, $21\bar{2}\bar{1}\underline{22} = a. 21\bar{2}\bar{1}2 + b. 21\bar{2}\bar{1}+ 21\bar{2}\bar{1}\bar{2}$.
Since $21\bar{2}\bar{1}\bar{2} = \bar{1}\bar{2}\bar{1} 21$ this proves (11).

We have $\underline{22}112212 = a. 2112212+b.112212+\bar{2}112212$,
then $\bar{2}\underline{11}2212 = a. \bar{2}12212 + b. \bar{2}2212 + \bar{2}\bar{1}2212$
and $\bar{2}\bar{1}2212 = a.\bar{2}\bar{1}212+ b \bar{2}\bar{1}12 + \bar{2}\bar{1}\bar{2}12$.
Finally, $\bar{2}\bar{1}\bar{2}12 = 12 \bar{1}\bar{2}\bar{1}$ and this proves (12).

We have $2112122121 = 211212\underline{21212}\bar{2}
= 21\underline{12121}2121\bar{2}= 2\underline{12121}22121\bar{2}
= 22121\underline{222}121\bar{2}
= a .2212122121\bar{2} + b. 221212121\bar{2} +  22121121\bar{2}$
and this proves (13).

We have $2122122\underline{11} = a.21221221+b.2122122+ 2122122\bar{1}$,
and $21221\underline{22}\bar{1} = a.212212\bar{1}+b.21221\bar{1}+ 21221\bar{2}\bar{1}$.
Now, $21221\bar{2}\bar{1} = 212\underline{21\bar{2}\bar{1}\bar{2}} 2 =  212\bar{1}\bar{2}\bar{1}21 2$
and this proves (14).

We expand  $2(11)(22)(11)(22)$ by using four times the relation $x^2 = a.x + b + x^{-1}$
for $x \in \{1 ,2 \}$ at the four places between parenthesis we get 
$211221122 = a.21122112+b.2112211+a.211221\bar{2}+b.2112+a.2112\bar{1}\bar{2}+b.21\bar{2}
+a.21\bar{2}\bar{1}\bar{2}+ b.\bar{1}\bar{2} + 2\bar{1} \bar{2}\bar{1} \bar{2}$. Now $\bar{1} \bar{2}\bar{1} \bar{2}1 = 2\bar{1} \bar{2}\bar{1} \bar{2}$ and this proves (15).

Rule $\# 16$ is similar to rule \# 15 : we expand $(22)(11)(22)(11)2$ and use $1\bar{2}\bar{1}\bar{2}\bar{1} = \bar{2}\bar{1}\bar{2}\bar{1}2$.
Rule $\# 17$ is
similar to rules \# 15 and \# 16 : expand $2112(11)(22)(11)$ and use $2112\bar{1}\bar{2}\bar{1} = 211\underline{2\bar{1}\bar{2}\bar{1}\bar{2}}2
= 211\bar{1}\bar{2}\bar{1}\bar{2}12 =  21\bar{2}\bar{1}\bar{2}12$.

By expanding $(22)(11)21121$
we get $221121121 = (a).21121121+(b).1121121+(a).\bar{2}121121+(b).1121+\bar{2}\bar{1} 21121$.
Since $\bar{2}\bar{1} 21121 = 1 \underline{\bar{1}\bar{2}\bar{1} 21}121
= 1 21\bar{2}\bar{1} \bar{2}121$ this proves (18).

By expanding $2122112(11)(22)$ we get
$21221121122 = (a).2122112122+(b).212211222+(a).2122112\bar{1}2+(b).2122112\bar{1}+2122112\bar{1}\bar{2}$
and
$2122112\bar{1}\bar{2} = 21221\underline{12\bar{1}\bar{2}\bar{1}}1
= 212\underline{21\bar{2}\bar{1}\bar{2}}121
= 212\bar{1}\bar{2}\bar{1}21121$ which proves (19).

By expanding $2(11)(22)121121$ we get
$21122121121 =  (a).2122121121+(b).222121121+(a).2\bar{1}2121121+(b).221121+2 \bar{1}\bar{2}121121$
and $2 \underline{\bar{1}\bar{2}121}121 = 2 212\underline{\bar{1}\bar{2}121}
= 2 212212\bar{1}\bar{2}$ which proves (20).

By expanding $22112(11)(22)12$
we get $22112112212 = (a).2211212212+(b).221122212+(a).22112\bar{1}212+(b).221122+22112\bar{1}\bar{2}12$
and $22112\bar{1}\bar{2}12=22112\underline{\bar{1}\bar{2}121}\bar{1}
=22112212\bar{1}\bar{2}\bar{1}$ proves (21).

We have $211211211212 = 21121121\underline{12121} \bar{1}
= 21121\underline{12121}212 \bar{1}
= 2112\underline{12121}2212 \bar{1}
= 2112212122212 \bar{1}$ and this proves (22).
We have $211211212212 = 21121\underline{12121}\bar{1}212
= 2112121212\bar{1}212$ and this proves (23).
We have $211212211211 = 21\bar{2}\underline{21212}211211= 21\bar{2}12121211211$ and this proves (24).
We expand $211211221(22)(11)$
and get $2112112212211 = (a).211211221211+(b).21121122111+(a).211211221\bar{2}1 +(b).211211221\bar{2}+211211221\bar{2}\bar{1}$
and
$211211221\bar{2}\bar{1} = 2112112\underline{21\bar{2}\bar{1}\bar{2}}2
= 21121\underline{12\bar{1}\bar{2}\bar{1}}212
= 211\underline{21\bar{2}\bar{1}\bar{2}}12212
= 211\bar{1}\bar{2}\bar{1}2112212
= 21\bar{2}\bar{1}2112212$ and this proves (25).
We have  $2112122122122 = 21\underline{12121}\bar{1}2122122
= 2121212\bar{1}2122122$ and this proves (26).

By expanding $2(11)(22)(11)211221$
we get $2112211211221 = (a).212211211221+(b).22211211221+(a).2\bar{1}211211221
 +(b).21211221+(a).2\bar{1}\bar{2}1211221+(b).21221+ 2 \bar{1}\bar{2}\bar{1}211221$
 and $2 \underline{\bar{1}\bar{2}\bar{1}21}1221 = 
 2 21\underline{\bar{2}\bar{1}\bar{2}12}21 = 
 2 2112\underline{\bar{1}\bar{2}\bar{1}21} = 
 2 211221\bar{2}\bar{1}\bar{2}$  and this proves (27).
 By expanding $2(11)(22)12212212$
we get $2112212212212 = (a).212212212212+(b).22212212212+(a).2\bar{1}212212212+(b).22212212+
2 \bar{1}\bar{2}12212212$, and $2 \bar{1}\bar{2}12212212 
= 2 \underline{\bar{1}\bar{2}121}\bar{1}212212 
= 2 212\bar{1}\bar{2}\bar{1}212212$ and this proves (28).
We have $2122122122122 = \bar{1}\underline{12121}\bar{1}2122122122
= \bar{1}21212\bar{1}2122122122$ and this proves (29). We have $2212212212212 = 2212212212\bar{1}\underline{12121}\bar{1}
= 2212212212\bar{1}21212\bar{1}$ and this proves (30).
We have $21121121121121=21121121121\bar{2}\underline{21212}\bar{2}
=21121121121\bar{2}12121\bar{2}$ and this proves (31).
By expanding $2112112112(11)(22)$ we get  $21121121121122  = 
a.2112112112122+b.211211211222+a.2112112112\bar{1}2+b.2112112112\bar{1}+2112112112\bar{1}\bar{2}$,
and $2112112112\bar{1}\bar{2} = 21121121\underline{12\bar{1}\bar{2}\bar{1}}1
= 21121121\bar{2}\bar{1}\bar{2}121$,
which proves (32).

We have $21121121122122 = 21\bar{2}\underline{21212}\bar{2}121122122
= 21\bar{2}12121\bar{2}121122122$ which proves (33).
We have $21122122112112 = 2112\bar{1}12121\bar{1}2112112
= 2112\bar{1}21212\bar{1}2112112$ which proves (34).
We have $211211221221221 = 
21\bar{2}21212\bar{2}1221221221
=21\bar{2}12121\bar{2}1221221221$ which proves (35).
We have $2112112112212112=2112112112\underline{21212}\bar{2}12
=211211211212121\bar{2}12$  which proves (36).

\section{Rules for $G_{21}$}

\label{sect:rulesG21}

\subsection{Semi-manual procedures}

Let $Y$ be the alphabet $\{ 1,2, \bar{1},\bar{2} \}$, $M(Y)$ the free monoid over $Y$,
and $F(Y) \subset M(Y)$ the subset of \emph{freely reduced words}, that is the set of natural
representatives of the free group on $\{1, 2 \}$ viewed as a quotient of $M(Y)$. We denote $M^+(Y) = M(\{1, 2 \}) \subset M(Y)$
the submonoid of \emph{positive words}. We let $\mathrm{red} : M(Y) \to F(Y)$ denote
the usual reduction procedure, and $\mathrm{red} : R M(Y) \to R F(Y)$ its natural linear extension, where we let $R M(Y)$ the monoid algebra over $R$ and $R F(Y)$ the (free) submodule spanned
by $F(Y)$. We define $\mathrm{pos} : R M(Y) \to R M(Y)$ and call \emph{positivation} the (unique) algebra morphism
mapping $1 \mapsto 1$, $2 \mapsto 2$, $\bar{2} \mapsto 22 - a.2 - b.\emptyset$,
$\bar{1} \mapsto 1 - q.\emptyset$.

A more complicated procedure is what we call \emph{expansion}. By convention we let $\overline{\bar{y}}=y$ for all $y \in \{ 1,2 \}$. For $I \subset \N^*$, let us define the \emph{$I$-inversion} map $\mathrm{inv}_I : M(Y) \to M(Y)$
as follows. If $y = y_1y_2y_3\dots y_n \in M(Y)$ is a word in $n$ letters, with $y_k \in Y$,
$\mathrm{inv}_I(y) = y' =y'_1y'_2y'_3\dots y'_n \in M(Y)$ is defined by $y'_k = \bar{y}_k$ if $k \in I$, $y'_k = y_k$ if $k \not\in I$.
We now define the partially defined \emph{expansion} map 
$\mathrm{exp}_I : M(Y) \dashrightarrow M(Y)$
 with respect to $I$ by induction on the cardinality of $I$.
If $I = \emptyset$, then $\exp_{\emptyset}$ is the identity map. If not, let $i_0 = \min (I)$, and let $J$ such that $I = J \sqcup \{ i_0 \}$.
If $y = y_1y_2y_3\dots y_n \in M(Y)$ is a word in $n$ letters, with $y_k \in Y$, then 
$\mathrm{exp}_I(y)$ is defined if $\mathrm{exp}_J(y)$ is defined, $n \geq i_0$ and if
\begin{itemize}
\item either $y_{i_0} = 1$, in which case $\exp_I(y) = q. y' + z$
with $y' = y'_1y'_2\dots y'_{n-1}$ where $y'_k = y_k$ for $k < i_0$, $y'_k = y_{k+1}$ for $k \geq i_0$, and
$z = \mathrm{exp}_J(\mathrm{inv}_{\{i_0 \}}(y))$
\item either $y_{i_0} = y_{i_0}+1=2$, in which case $\exp_I(y) = a.y' + b. y'' + z$
with 
\begin{itemize}
\item $y'' = y''_1y''_2\dots y''_{n-2}$ where $y''_k = y_k$ for $k < i_0$, $y''_k = y_{k+2}$ for $k \geq i_0$
\item $y' = y'_1y'_2\dots y'_{n-1}$ where $y'_k = y_k$ for $k < i_0$, $y'_{i_0} = y_{i_0}=2$, $y'_k = y_{k+1}$ for $k \geq i_0+1$.
\item $z = \mathrm{exp}_J(\mathrm{inv}_{\{i_0 \}}(y''))$.
\end{itemize}
\end{itemize}
It is easily checked that, when defined, $\mathrm{exp}_I(y) = \mathrm{tail}_I(y)+ \mathrm{head}_I(y)$
with $\mathrm{head}_I(y) \in M(Y)$
being characterized, with the above notations, by $\mathrm{head}_{\{ i_0\} \sqcup J }(y) =   \mathrm{head}_J(z)$,
and $\mathrm{head}_{\emptyset}(y) = y$.

\subsection{Rules}

We can now give the set of rules for $G_{21}$, the justification that they
correspond to genuine relations inside its Hecke algebra basically relying on the above sections.

$$
\begin{array}{clcl}
(1) & 11 & \leadsto & (q).1+\emptyset \\
(2) & 222 & \leadsto & (a).22+(b).2+\emptyset \\
(3) & 2121212121 & \leadsto & 1212121212 \\
(4) & 21212121221 & \leadsto & \redu(\queue_{11}( \bar{1}1*w) + \bar{1}\bar{2}1212121211) \\
(5) & 221221212121 & \leadsto &  \redu(\queue_{1,3,4}(w) + 1212121\bar{2}\bar{1}\bar{2}) \\
(6) & 212121221221 & \leadsto & \redu(\queue_{7,9,10,12}(212121221221)+\overline{1212}121212) \\
(7) & 2121221221221  & \leadsto & \redu(\queue_{5,7,8,10,11,13}(2121221221221)+\overline{121212}1212) \\
(8) & 2212212212121 & \leadsto & \redu(\queue_{1,3,4,6,7}(2212212212121)+12121\bar{2}\bar{1}\bar{2}\bar{1}\bar{2}) \\
(9) & 212122121221221 & \leadsto & \redu(\queue_{10,12,13,15}(w*\bar{2}\bar{1}12)+\bar{1}\bar{2}\bar{1}\bar{2}\bar{1}2121121212) \\
(10) & 22121221212121 & \leadsto&  \redu(\queue_{5,6,8}(w*2\bar{2})+22122121212\bar{1}\bar{2}\bar{1}\bar{2}) \\
\end{array}
$$
{}
$$
\begin{array}{clcl}
(11) & 21212122121221 & \leadsto& \redu(\queue_{11,13}(\bar{1}\bar{2}21*w)+\queue_{12,14,15}(w')+w'') \\
& & & w' = \bar{1}\bar{2}\bar{1}\bar{2}1212121221221 \\
& & & w'' = \bar{1}\bar{2}\bar{1}\bar{2}\bar{2}\bar{1}\bar{2}12121211 \\
(12) & 21221221221221 & \leadsto & \redu(\queue_{3,5,6,8,9,11,12,14}(w)+\bar{1}\bar{2}\bar{1}\bar{2}\bar{1}\bar{2}\bar{1}\bar{2}12) \\
(13) & 22122122122121 & \leadsto & \redu(\queue_{1,3,4,6,7,9,10,12}(w)+12\bar{1}\bar{2}\bar{1}\bar{2}\bar{1}\bar{2}\bar{1}\bar{2}) \\
(14) & 2121212212121221 & \leadsto & \redu(\queue_{14,16}(w*\bar{2}\bar{1}12)+2121212\bar{1}\bar{2}\bar{1}\bar{2}12121212) \\
(15) & 221221212212121 & \leadsto & \redu(\queue_{8,9}(w*212\bar{2}\bar{1}\bar{2})+221221221212121\bar{2}\bar{1}\bar{2}\bar{1}\bar{2}) \\
(16) & 2212121221212121 & \leadsto & \redu(\queue_{7,8}(w*2\bar{2})+22121221212121\bar{2}\bar{1}\bar{2}) \\
(17) & 2212212212122121 & \leadsto & \redu(\queue_{3,5,6,8,9}(1\bar{1}*w)+12121\bar{2}\bar{1}\bar{2}\bar{1}\bar{2}\bar{1}2121) \\
(18) & 22122121212212121 & \leadsto & \redu(\queue_{3,5,6}(1\bar{1}*w)+1212121\bar{2}\bar{1}\bar{2}\bar{1}212121) \\
(19) & 212122122121221221 & \leadsto & \redu(\queue_{12,13,15,16,18}(w*\bar{2}\bar{1}12)+\bar{1}\bar{2}\bar{1}\bar{2}\bar{1}21221221212) \\
(20) &221221212212212121  & \leadsto & \redu(\queue_{8,9,11,12}(w*2\bar{2})+2212212212121\bar{2}\bar{1}\bar{2}\bar{1}\bar{2}) \\
\end{array}
$$
{}
$$
\begin{array}{clcl}
(21) & 2212122121212212121 & \leadsto & \redu(\queue_{5,6}(\hat{w})+221221212121\bar{2}\bar{1}\bar{2}\bar{1}212121) \\
 & & & \hat{w} = 221212212121212\bar{2}\bar{1}212121 \\
(22) & 2212212212122122121 & \leadsto & \redu(\queue_{3,5,6,8,9}(1\bar{1}*w)+12121\bar{2}\bar{1}\bar{2}\bar{1}\bar{2}\bar{1}2122121) \\
(23) & 22121221221212122121 & \leadsto & \redu(\queue_{5,6,8,9}(w)+\queue_{5,7,8}(w'))+w'' \\
& & & w' = 12\bar{2}\bar{1}2212212121\bar{2}\bar{1}\bar{2}\bar{1}2121 \\
& & & w'' =1212121\bar{2}\bar{1}\bar{2}\bar{1}\bar{2}\bar{2}\bar{1}\bar{2}\bar{1}2121 \\
(24) & 22122121221212212121 & \leadsto & \redu(\queue_{13,14}(w*212\bar{2}\bar{1}\bar{2})+22122121221221212121\bar{2}\bar{1}\bar{2}\bar{1}\bar{2}) \\
(25) & 221221212212212122121 & \leadsto & \redu(\queue_{7,9,10,12}(121\bar{1}\bar{2}\bar{1}*w)+121212212212\bar{1}\bar{2}\bar{1}\bar{2}\bar{1}2121) \\
(26) & 2121221212212122121221 & \leadsto & \redu(\queue_{20,22}(w*\bar{2}\bar{1}\bar{2}\bar{1}1212)+21212\bar{1}\bar{2}\bar{1}\bar{2}\bar{1}21212212112121212) \\
(27) & 22121221221212212122121 & \leadsto & \redu(\queue_{5,6,8,9}(\hat{w})+2212212121121221212\bar{1}\bar{2}\bar{1}\bar{2}\bar{1}\bar{2}) \\
& & & \hat{w} = 2212122122121212\bar{2}\bar{1}212122121 \\
(28) & 22122121221212212122121  & \leadsto & \redu(\queue_{7,9,10}(121\bar{1}\bar{2}\bar{1}*w)+12121211212\bar{1}\bar{2}\bar{1}\bar{2}\bar{1}212122121) \\
(29) & 2212122121221212212121 & \leadsto & \redu(\queue_{15,16}(w*212\bar{2}\bar{1}\bar{2})+\queue_{10,11,13,14}(w')+w'') \\
& & & w' = 2212122121221221212121\bar{2}\bar{1}\bar{2}\bar{1}\bar{2}\\
& & & w'' = 221212212212121\bar{2}\bar{1}\bar{2}\bar{2}\bar{1}\bar{2}\bar{1}\bar{2} \\
(30) & 2212122121221212122121 & \leadsto & \redu(\queue_{14,15}(\hat{w})+\queue_{5,6,8,9}(w')+\queue_{5,7,8}(w'')+w''') \\
&&& \hat{w} = 22121221221\bar{1}\bar{2}1221212122121 \\
&&& w' = 22121221221212121\bar{2}\bar{1}\bar{2}\bar{1}2121 \\
&&& w'' = 12\bar{2}\bar{1}2212212121\bar{2}\bar{1}\bar{2}\bar{2}\bar{1}\bar{2}\bar{1}2121 \\
&&& w''' = 1212121\bar{2}\bar{1}\bar{2}\bar{1}\bar{2}\bar{2}\bar{1}\bar{2}\bar{2}\bar{1}\bar{2}\bar{1}2121
\end{array}
$$

\begin{table}
$$
\begin{array}{|cl||cl||cl|} 
\hline
\mathrm{Num.} & \mathrm{Word} & \mathrm{Num.} & \mathrm{Word} & \mathrm{Num.} & \mathrm{Word} \\ 
 \hline 
 \hline 
 1 & 111  & 13 & 2112122121  & 25 & 2112112212211  \\ 
\hline 
2 & 222 &  14 & 212212211  & 26 & 2112122122122  \\ 
\hline 
3 & 21212  & 15 & 211221122  & 27 & 2112211211221  \\ 
\hline 
4 & 2112121  & 16 & 221122112  & 28 & 2112212212212  \\ 
\hline 
5 & 2121122  & 17 & 2112112211  & 29 & 2122122122122  \\ 
\hline 
6 & 22122121  & 18 & 221121121  & 30 & 2212212212212  \\ 
\hline 
7 & 2211212  & 19 & 21221121122  & 31 & 21121121121121  \\ 
\hline 
8 & 21211211  & 20 & 21122121121  & 32 & 21121121121122  \\ 
\hline 
9 & 212112122  & 21 & 22112112212  & 33 & 21121121122122  \\ 
\hline 
10 & 221211212  & 22 & 211211211212  & 34 & 21122122112112  \\ 
\hline 
11 & 21221122  & 23 & 211211212212  & 35 & 211211221221221  \\ 
\hline 
12 & 22112212  & 24 & 211212211211  & 36 & 2112112112212112  \\ 
\hline 
\end{array} 
$$
\caption{Dominant terms of the Gr\"obner basis for $G_{20}$}
\end{table}

\begin{table}
\resizebox{16cm}{!}{
$
\begin{array}{|r|r|r|r|r|r|r|r|r|}
\hline
\emptyset  & 112122 & 2112212 & 21221121 & 211212212 & 2112112122 & 12211211221 & 121122122122 & 2112112112212\\
1 & 112211 & 2121121 & 21221211 & 211221121 & 2112112212 & 12212211211 & 121221221221 & 2112112122112\\
2 & 112212 & 2122112 & 21221221 & 211221211 & 2112122112 & 12212212212 & 122122112112 & 2112112212112\\
11 & 121121 & 2122121 & 22112112 & 211221221 & 2112122122 & 21121121121 & 122122122122 & 2112112212212\\
12 & 121122 & 2122122 & 22112211 & 212211211 & 2112211211 & 21121121122 & 211211211211 & 2112212211211\\
21 & 121211 & 2211211 & 22121121 & 212212112 & 2112212112 & 21121121221 & 211211211221 & 11211211211211\\
22 & 121221 & 2211221 & 22122112 & 212212212 & 2112212211 & 21121122121 & 211211212211 & 11211211211221\\
112 & 122112 & 2212112 & 22122122 & 221121122 & 2112212212 & 21121122122 & 211211221211 & 11211211212211\\
121 & 122121 & 2212211 & 112112112 & 221221121 & 2122112112 & 21121221121 & 211211221221 & 11211211221211\\
122 & 122122 & 2212212 & 112112122 & 221221221 & 2122121121 & 21121221221 & 211212212212 & 11211211221221\\
211 & 211211 & 11211211 & 112112211 & 1121121121 & 2122122122 & 21122112112 & 211221121122 & 11211212212212\\
212 & 211212 & 11211212 & 112112212 & 1121121122 & 2211211221 & 21122122112 & 211221221121 & 11211221121122\\
221 & 211221 & 11211221 & 112121121 & 1121121221 & 2212211211 & 21122122122 & 211221221221 & 11211221221121\\
1121 & 212112 & 11212112 & 112122112 & 1121122112 & 2212212212 & 21221221221 & 212212212212 & 11211221221221\\
1122 & 212211 & 11212211 & 112122121 & 1121122121 & 11211211211 & 22122112112 & 221221221221 & 11212212212212\\
1211 & 212212 & 11212212 & 112122122 & 1121122122 & 11211211212 & 22122122122 & 1121121121121 & 11221221221221\\
1212 & 221121 & 11221121 & 112211211 & 1121211212 & 11211211221 & 112112112112 & 1121121121122 & 12112112112112\\
1221 & 221122 & 11221122 & 112211221 & 1121221121 & 11211212211 & 112112112122 & 1121121121221 & 12112112112212\\
2112 & 221211 & 11221211 & 112212112 & 1121221211 & 11211212212 & 112112112212 & 1121121122121 & 12112112122112\\
2121 & 221221 & 11221221 & 112212211 & 1121221221 & 11211221121 & 112112122112 & 1121121122122 & 12112112212112\\
2122 & 1121121 & 12112112 & 112212212 & 1122112112 & 11211221211 & 112112122122 & 1121121221121 & 12112112212212\\
2211 & 1121122 & 12112122 & 121121121 & 1122112211 & 11211221221 & 112112211211 & 1121121221221 & 12112212211211\\
2212 & 1121211 & 12112211 & 121121122 & 1122121121 & 11212211211 & 112112212112 & 1121122112112 & 21121121122121\\
11211 & 1121221 & 12112212 & 121121221 & 1122122112 & 11212212112 & 112112212211 & 1121122122112 & 21121121221121\\
11212 & 1122112 & 12121121 & 121122112 & 1122122122 & 11212212212 & 112112212212 & 1121122122122 & 21121122122122\\
11221 & 1122121 & 12122112 & 121122121 & 1211211211 & 11221121122 & 112122112112 & 1121221221221 & 112112112112112\\
12112 & 1122122 & 12122121 & 121122122 & 1211211212 & 11221221121 & 112122121121 & 1122122112112 & 112112112112212\\
12121 & 1211211 & 12122122 & 121211212 & 1211211221 & 11221221221 & 112122122122 & 1122122122122 & 112112112122112\\
12122 & 1211212 & 12211211 & 121221121 & 1211212211 & 12112112112 & 112211211221 & 1211211211211 & 112112112212112\\
12211 & 1211221 & 12211221 & 121221211 & 1211212212 & 12112112122 & 112212211211 & 1211211211221 & 112112112212212\\
12212 & 1212112 & 12212112 & 121221221 & 1211221121 & 12112112212 & 112212212212 & 1211211212211 & 112112212211211\\
21121 & 1212211 & 12212211 & 122112112 & 1211221211 & 12112122112 & 121121121121 & 1211211221211 & 121121121122121\\
21122 & 1212212 & 12212212 & 122112211 & 1211221221 & 12112122122 & 121121121122 & 1211211221221 & 121121121221121\\
21211 & 1221121 & 21121121 & 122121121 & 1212211211 & 12112211211 & 121121121221 & 1211212212212 & 121121122122122\\
21221 & 1221122 & 21121122 & 122122112 & 1212212112 & 12112212112 & 121121122121 & 1211221121122 & 211211211221211\\
22112 & 1221211 & 21121221 & 122122122 & 1212212212 & 12112212211 & 121121122122 & 1211221221121 & 1121121121122121\\
22121 & 1221221 & 21122112 & 211211211 & 1221121122 & 12112212212 & 121121221121 & 1211221221221 & 1121121121221121\\
22122 & 2112112 & 21122121 & 211211212 & 1221221121 & 12122112112 & 121121221221 & 1212212212212 & 1121121122122122\\
112112 & 2112122 & 21122122 & 211211221 & 1221221221 & 12122121121 & 121122112112 & 1221221221221 & 1211211211221211\\
112121 & 2112211 & 21211212 & 211212211 & 2112112112 & 12122122122 & 121122122112 & 2112112112112 & 11211211211221211\\
\hline
\end{array}
$
}
\caption{Basis for $G_{20}$}
\end{table}

\begin{table}
$$
\begin{array}{|c|l||c|l|} 
\hline
\mathrm{Num.} & \mathrm{Word}  & \mathrm{Num.} & \mathrm{Word}  \\ 
\hline
1 & 11 & 16 & 2212121221212121 \\ 
\hline 
2 & 222 & 17 & 2212212212122121 \\ 
\hline 
3 & 2121212121 & 18 & 22122121212212121 \\ 
\hline 
4 & 21212121221 & 19 & 212122122121221221 \\ 
\hline 
5 & 221221212121 & 20 & 221221212212212121 \\ 
\hline 
6 & 212121221221 & 21 & 2212122121212212121 \\ 
\hline 
7 & 2121221221221 & 22 & 2212212212122122121 \\ 
\hline 
8 & 2212212212121 & 23 & 22121221221212122121 \\ 
\hline 
9 & 212122121221221 & 24 & 22122121221212212121 \\ 
\hline 
10 & 22121221212121 & 25 & 221221212212212122121 \\ 
\hline 
11 & 21212122121221 & 26 & 2121221212212122121221 \\ 
\hline 
12 & 21221221221221 & 27 &  22121221221212212122121 \\ 
\hline 
13 & 22122122122121 & 28 &  22122121221212212122121 \\ 
\hline 
14 & 2121212212121221 & 29 & 2212122121221212212121  \\ 
\hline 
15 &  221221212212121 & 30 & 2212122121221212122121 \\ 
\hline 
\end{array} 
$$
\caption{Dominant terms of the Gr\"obner basis for $G_{21}$}
\end{table}


\begin{thebibliography}{99999}
\bibitem{ARIKI} S. Ariki, {\it Representation theory of a Hecke algebra of $G(r,p,n)$}, J. Algebra {\bf 177} (1995), 164--185.
\bibitem{ARIKIKOIKE} S. Ariki, K. Koike, {\it A Hecke algebra of $(\Z/r\Z)\wr \mathfrak{S}_n$
and construction of its irreducible representations}, Advances in Math. {\bf 106} (1994), 216–243.
\bibitem{BMR} M. Brou\'e, G. Malle, R. Rouquier, {\it Complex reflection groups, braid groups, Hecke algebras}, J. Reine Angew. Math. {\bf 500} (1998) 127-190.
\bibitem{CHAVLITHESE} E. Chavli, {\it The BMR freeness conjecture for
exceptional groups of rank 2}, doctoral thesis, Univ. Paris Diderot (Paris 7), 2016.
\bibitem{CHAVLICRAS} E. Chavli, 
{\it The BMR freeness conjecture for the first two families of the exceptional groups of rank 2},
to appear in Comptes Rendus Math\'ematiques.
\bibitem{CB3} E. Chavli, {\it Universal deformations of the finite quotients of the braid group on 3 strands}, preprint 2015, to appear in J. Algebra.
\bibitem{GBNP} A.M. Cohen, D.A.H. Gijsbers,  and J.W. Knopper, {\it GBNP package version 1.0.1}, \url{http://mathdox.org/gbnp/}.
\bibitem{ETINGOFRAINS} P. Etingof, E. Rains, {\it Central extensions of preprojective algebras,
the quantum Heisenberg algebra, and 2-dimensional complex reflection groups}, J. Algebra {\bf 299} (2006), 570--588.
 \bibitem{GP} M. Geck, G. Pfeiffer, {\it Characters of finite Coxeter
 groups and Iwahori-Hecke algebras}. London Mathematical Society Monographs. New Series, 21. The Clarendon Press, Oxford University Press, New York, 2000. 
 \bibitem{CYCLO} I. Marin, {\it The freeness conjecture for Hecke algebras of complex reflection groups, and the case of the Hessian group $G_{26}$}, J. Pure Applied Algebra {\bf 218} (2014) 704-720.
\bibitem{HECKECUBIQUE} I. Marin, {\it The cubic Hecke algebra on at most 5 strands},  J. Pure Appl. Algebra {\bf 216} (2012), 2754–2782.
\bibitem{KRAMCRG} I. Marin, {\it Krammer representations for complex braid groups},
J. Algebra {\bf 371} (2012), 175--206.
\bibitem{CYCLO} I. Marin, {\it The freeness conjecture for Hecke algebras of complex reflection groups, and the case of the Hessian group $G_{26}$}, J. Pure Appl. Algebra {\bf 218} (2014),  704–720.
\bibitem{MARINPFEIFFER} I. Marin, G. Pfeiffer, {\it The BMR freeness conjecture for the 2-reflection groups}, preprint 2014, to appear in Math. of Comput.
\end{thebibliography}
\end{document}